\documentclass[11pt,a4paper]{article}

\usepackage{amsfonts}
\usepackage{amsmath}
\usepackage{amssymb}
\usepackage{latexsym} 
\usepackage{tabularx}
\usepackage{graphicx}
\usepackage{xcolor}
\usepackage{xspace}
\usepackage[figure,tworuled,linesnumbered,vlined]{algorithm2e}

\newcommand{\TOLTO}[1]{}

\newenvironment{proof}{{\bf Proof.\/}}{\hfill $\Box$\vskip 0.1in}

\newtheorem{theorem}{Theorem}[section] 
\newtheorem{definition}{Definition}[section] 
\newtheorem{lemma}[theorem]{Lemma} 
\newtheorem{corollary}[theorem]{Corollary}

\newtheorem{remark}[theorem]{Remark}

\newcommand{\tf}[1]{\ensuremath{{\mathit{tf}}(#1)}}
\newcommand{\gp}[1]{\ensuremath{gp(#1)}} 
\newcommand{\conv}[1]{\ensuremath{hull( #1 )}}
\newcommand{\sq}{\ensuremath{\hspace{2pt}\square\hspace{2pt}}}

\newcommand{\mvs}{mutual-visibility set\xspace}
\newcommand{\Mvs}{Mutual-visibility set\xspace}
\newcommand{\mvn}{mutual-visibility number\xspace}
\newcommand{\mvp}{{\sc Mutual-Visibility}\xspace}
\newcommand{\mva}{{\sc MV}\xspace}
\newcommand{\bfs}{{\sc BFS\_MV}\xspace}
\newcommand{\mvt}{{\sc Mutual-Visibility Test}\xspace}
\newcommand{\sat}{{\sc 3SAT}\xspace}

\definecolor{linecomment}{rgb}{0.95, 0.1, 0.1}

\title{\bf Mutual Visibility in Graphs%
}

\author{ 
{\sc Gabriele Di Stefano}
}

\date{  {\small 
        Department of Information Engineering, Computer Science and Mathematics \\[1mm] 
        University of L'Aquila -- Italy\\[1mm]
                gabriele.distefano@univaq.it
        } 
}

\begin{document}

\maketitle


\begin{abstract}
Let $G=(V,E)$ be a graph and $P\subseteq V$ a set of \emph{points}. Two points are \emph{mutually visible} if there is a shortest path between them without further points. $P$ is a \emph{\mvs} if its points are pairwise mutually visible.
The \emph{\mvn} of $G$ is the size of any largest \mvs. In this paper we start the study about this new invariant and the {\mvs}s in undirected graphs. We introduce the \mvp problem which asks to find a \mvs with a size larger than a given number. We show that this problem is  NP-complete, whereas, to check whether a given set of points is a \mvs is solvable in polynomial time. Then we study {\mvs}s and {\mvn}s on special classes of graphs, such as block graphs, trees, grids, tori, complete bipartite graphs, cographs. We also provide some relations of the \mvn of a graph with other invariants.
\end{abstract}



\section{Introduction}


Given a set of points in Euclidean space, they are mutually visible if and only if   no three of them are collinear. Then two points $p$ and $q$ are mutually visible when no further point is on the line segment $pq$.
A line segment in Euclidean space represents the shortest path between two points, but in more general topologies, this type of path (called geodesic) may not be unique.
Then, in general, two points are mutually visible when there exists at least a shortest path between them without further points.

In this paper, we investigate the mutual visibility of a set of points in topologies represented by graphs (e.g., see Figure~\ref{fig:tori}b).  
In particular, a fundamental problem is finding the maximum number of points in mutual visibility that a given graph can have. 
To this aim, consider the following new invariant: the \emph{\mvn} of a graph is the size of any largest \emph{\mvs}, that is, a subset of the vertices  (\emph{points}) that are in mutual visibility. 
To study this invariant from a computational point of view, we introduce the \mvp problem: find a \mvs with a size larger than a given number. We prove that this problem is NP-complete, whereas, to check whether the points of a given set are in mutual visibility is a problem solvable in $O(n^3)$ time for graphs with $n$ vertices.  Then, given this situation, our work proceeds by investigating the {\mvn} for special classes of graphs, showing how the \mvp problem can be solved in polynomial time. We also provide some relations of the \mvn of a graph with other invariants.

While these new concepts are interesting in themselves, their study is motivated by the fundamental role that mutual visibility plays in problems arising in the context of mobile entities, as shown below. Furthermore, points of a graph in mutual visibility may represent  entities on some nodes of a computer/social network that want to communicate in a efficient and ``confidential'' way, that is, in such a way that the exchanged messages  do not pass through other entities.

\paragraph{Related works}
Questions about sets of points and their mutual visibility in Euclidean plane have been investigated since the end of XIX century. Perhaps, the most famous problem was posed by Sylvester~\cite{sylvester1893}, who conjectured that it is not possible to arrange a finite set of points ``so that a right line through every two of them shall pass through a third, unless they all lie in the same right line''. A correct proof was given by Gallai~\cite{gallai44} some 40 years later, with a theorem now known as Sylvester--Gallai theorem.  In~\cite{dudeney917} Dudeney posed the celebrated and still open no-three-in-line problem:  find the maximum number of points that
can be placed in an $n\times n$  grid so that no three points lie on a line. 
In~\cite{hardy08}, Chapter III, it is shown how to place a set of points with integer positive coordinates $(i,j)$, $j\leq i$,  in such a way that each point is in mutual visibility with the origin $(0,0)$, by also maximizing the number of points with the same abscissa. 
This disposition shows interesting relations with the Farey series and the Euler's totient function $\phi$: the number of points with abscissa $n$ is exactly $\phi(n)$.

More recently, mutual visibility has been studied in the context of mobile entities modeled as points in the Euclidean plane, whose visibility can be obstructed by the presence of other mobile entities. The problem investigated in~\cite{DiLuna17} is perhaps the most basic:
starting from arbitrary distinct positions in the plane, within finite time the mobile entities must reach a configuration in which they are in distinct locations and they can all see each other. Since then, many papers have addressed the same subject (e.g., see~\cite{Aljohani18a,Bhagat20,Poudel2021,Sharma18}) and similar visibility problems were considered in different contexts where the entities are ``fat robots'' modeled as disks in the plane (e.g., see~\cite{Poudel19}) or are points on a grid based terrain and their movements are restricted only along grid lines (e.g., see~\cite{Adhikary18}).

Visibility problems were also studied on graphs. Wu and Rosenfeld~\cite{Rosenfeld94} considered the mutual visibility in pebbled graphs. 
They assumed that the visibility may be obstructed  by ``pebbles" placed on some vertices of the graph. Two unpebbled vertices $u,v$ of a pebbled graph $G$ are mutually visible if and only if there exists a shortest path $p$ in $G$ between $u$ and $v$ such that no vertex of $p$ is pebbled. 
In~\cite{Wu98} they consider edge pebblings and vertex pebblings, 
by showing that the visibility relations defined by edge and vertex pebblings are incomparable. 

Motivated by the Dudeney's problem, in~\cite{manuel:18} the {\sc General Position} problem was introduced. Few years before the same problem was posed in~\cite{chand:16}. A subset $S$ of vertices in a graph $G$ is a \emph{general position set} if no triple of vertices from $S$ lie in a common geodesic in $G$. The {\sc General Position} problem is to find a largest general position set of $G$, the order of such a set is the \emph{general position number} $\gp{G}$. Since its introduction, the general position number has been studied for several graph classes (e.g., grid networks~\cite{manuel:18b}, cographs and bipartite graphs~\cite{acckt:19},  graph classes with large general position number~\cite{tc:20}, Cartesian products of graphs~\cite{kpry:21}).

The difference between a general position set $S$ and a \mvs $P$ is that two vertices are in $P$ if there is a shortest path between them with no further vertex in $P$, whereas two vertices are in $S$ if for \emph{every} shortest path between them no further vertex is in $S$. The two concepts are intrinsically different, but closely related, since the vertices of a general position set are in mutual visibility.  

Again in the context of mobile entities, in~\cite{Aljohani18b} it is studied the
{\sc Complete Visitability}  problem of repositioning a given number of robots on the vertices of a graph so that each robot has a path to all others without
visiting an intermediate vertex occupied by any other robot.
Here, the required paths are not shortest paths and the studied graphs are restricted to the infinite squared grid and the infinite hexagonal grid, both embedded in the Euclidean plane.

\paragraph{Contribution}
In Section~\ref{sec:notation},  formal definitions of \mvs and \mvn are provided along with basic notations and some preliminary results. Algorithmic results about the \mvp problem are shown in Section~\ref{sec:complexity}. In Section~\ref{sec:graphs} we study the {\mvs}s and {\mvn}s for special classes of graphs. Comparisons between general position numbers and {\mvn}s for certain graph classes are provided in Sections~\ref{sec:notation} and~\ref{sec:graphs}. Concluding remarks and notes about further studies on the subject are provided in Section~\ref{sec:concl}.



\section{Notation and preliminaries}\label{sec:notation}
In this work we consider finite, simple, loopless, undirected and unweighted
graphs $(V,E)$ with vertex set $V$ and  edge set $E$. We use standard 
terminologies from~\cite{graph_classes_survey,graph_theory}, some of which are
briefly reviewed here.

\paragraph{Basic notation.} 
Let $G=(V,E)$ be a graph.  A {\em subgraph} of $G$ is a graph having all its vertices
and edges in $G$. 
Given a subset $S$ of $V$, the  {\em induced subgraph} $G[S]$ of $G$ 
is the maximal subgraph of $G$ with vertex set $S$.  The subgraph of $G$ induced by $V\setminus S$ is denoted by $G-S$, and $G-x$
stands for $G-\{x\}$. 
If $v$ is a vertex of $G$, by $N_G(v)$ we denote the {\em neighbors} of 
$v$,  that is, the set of vertices that are adjacent to $v$, and 
by $N_G[v]$ we denote the {\em closed neighbors} of $v$, that is $N_G(v)\cup
\{v\}$. 
The number of edges incident to a vertex $v$ of a graph $G$ is the \emph{degree} of that vertex and is denoted $deg_G(v)$. Then $deg_G(v)=|N_G(v)|$ and
the maximum degree 
is denoted $\Delta(G)$.
 If $|N_G(v)|=1$, $v$ is called \emph{pendant} vertex.
Two vertices $u,v$ are \emph{true twins} if $uv \in E$ and $N_G[u]=N_G[v]$ and are \emph{false twins} if $uv \not \in E$ and $N_G(u)=N_G(v)$.
The operation of extending a graph by adding a new vertex which has a twin
in the obtained graph, is called \emph{splitting}~\cite{bandelt/mulder:86}.


A sequence of pairwise distinct vertices $(x_0, x_1,\ldots, x_n)$ is a {\em 
path} in $G$ if $x_{i}x_{i+1} \in E$ for $0\leq i < n$, and is an  {\em induced
path} 
if $G[\{x_0, \ldots, x_n\}]$ has $n$ edges. The \emph{length} of an induced
path is the number of its edges.  A {\em cycle}  in $G$ is a path $(x_0, \ldots, x_{n-1})$, $n\geq 3$, where also $x_{0}x_{n-1}\in E$. A \emph{$(x,y)$-path} is a path from $x$ to $y$. A graph $G$ is {\em connected} 
if for each pair of vertices $x$ and $y$ of  $G$ there is a $(x,y)$-path in $G$. In a connected graph $G$, the length of a shortest $(x,y)$-path  is called {\em distance} and is denoted by $d_G(x,y)$. The longest distance in a graph is its \emph{diameter}. A \emph{connected component} of $G$ is a maximal connected subgraph of $G$.
A vertex $x$ is an \emph{articulation vertex} if $G-x$ has more connected components than $G$. A graph $G=(V,E)$ is \emph{biconnected}  if $G-x$ is connected, for each $x\in V$.

A subgraph $H$ of $G=(V,E)$ is said to be \emph{convex} if all shortest paths
in $G$ between vertices of $H$ actually belong to $H$. The \emph{convex hull} of a subset $V'$ of vertices – denoted \emph{\conv{V'}} – is defined as the smallest convex subgraph containing $V'$. The \emph{hull number} $h(G)$ is the minimum cardinality among the subsets $V'$ of $V$ with $\conv{V'} =G$.

\paragraph{Operations on graphs} If $G$ is a graph, $\overline{G}$ denotes its \emph{complement}, that is the graph on the same vertices such that two distinct vertices of $\overline{G}$ are adjacent if and only if they are not adjacent in $G$. 
Given two graphs $G_1=(V_1,E_1)$ and $G_2=(V_2,E_2)$, such that $V_1\cap
V_2=\emptyset$, the \emph{disjoint union} $G_1\cup G_2$ denotes the graph
$(V_1\cup V_2, E_1\cup E_2)$; the \emph{join} $G_1 + G_2$ denotes the graph consisting in $G_1\cup G_2$ and all edges joining $V_1$ with $V_2$, that is $(V_1\cup V_2, E_1\cup E_2\cup\{xy~|~x\in V_1, y \in V_2\})$. To define the \emph{Cartesian product} $G_1 \sq G_2=(V,E)$, consider any two vertices $u=(u_1, u_2)$ and $v = (v_1, v_2)$ in $ V = V_1 \times V_2$. Then $uv\in E$  whenever either $u_1 = v_1$  and $u_2v_2\in E_2$ or $u_2 = v_2$ and $u_1v_1 \in E_1$.
We call $G_1$ and $G_2$ \emph{isomorphic}, and
write $G_1\sim G_2$ if there exists a bijection $\varphi :V_1 \rightarrow V_2$
with $xy \in E_1 \iff \varphi(x)\varphi(y) \in E_2$ for all $x,y \in
V_1$. 
%
\paragraph{Special graphs} In this paper we use
some special graphs. $K_n$ denotes the \emph{complete graph} (or \emph{clique})
with $n$ vertices and $n(n-1)/2$ edges. The \emph{clique number} $\omega(G)$ of a graph $G$ is the number of vertices in a maximum clique in $G$. $P_n$ denotes the  \emph{path graph} with $n$ vertices and
$n-1$ edges. $C_n$ denotes the  \emph{cycle graph} with $n$ vertices and
$n$ edges. Finally, $K_{m,n}=\overline{K_m}+\overline{K_n}$ denotes the \emph{complete bipartite graph}. A \emph{tree} is a connected graph without cycles and its pendant vertices are called \emph{leaves}.  The tree $K_{1,n}$ is called \emph{star} and can be obtained by adding $n$ pendant vertices to a
single vertex, called the \emph{center} of the star.
The graph $C_3$ is also called \emph{triangle}. 
A \emph{grid graph} $\Gamma_{m,n}=P_m\sq P_n$ is the Cartesian product of two paths $P_m$ and $P_n$. For $m\geq 3$ and $n\geq 3$ a graph $T_{m,n} =C_m \sq C_n$  obtained by the Cartesian product of two cycle graphs is called \emph{torus}.
 A connected graph obtained from $K_1$ by a sequence of splittings is called \emph{cograph}.

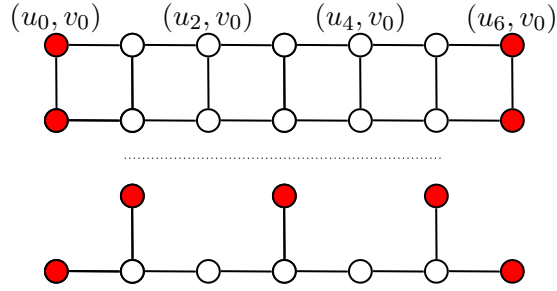
\begin{figure}[t]
   \graphicspath{{fig/}} 
 \begin{center}
    {\scalebox{1.0}{
\begingroup%
  \makeatletter%
  \providecommand\color[2][]{%
    \errmessage{(Inkscape) Color is used for the text in Inkscape, but the package 'color.sty' is not loaded}%
    \renewcommand\color[2][]{}%
  }%
  \providecommand\transparent[1]{%
    \errmessage{(Inkscape) Transparency is used (non-zero) for the text in Inkscape, but the package 'transparent.sty' is not loaded}%
    \renewcommand\transparent[1]{}%
  }%
  \providecommand\rotatebox[2]{#2}%
  \newcommand*\fsize{\dimexpr\f@size pt\relax}%
  \newcommand*\lineheight[1]{\fontsize{\fsize}{#1\fsize}\selectfont}%
  \ifx\svgwidth\undefined%
    \setlength{\unitlength}{215.77636552bp}%
    \ifx\svgscale\undefined%
      \relax%
    \else%
      \setlength{\unitlength}{\unitlength * \real{\svgscale}}%
    \fi%
  \else%
    \setlength{\unitlength}{\svgwidth}%
  \fi%
  \global\let\svgwidth\undefined%
  \global\let\svgscale\undefined%
  \makeatother%
  \begin{picture}(1,0.47953667)%
    \lineheight{1}%
    \setlength\tabcolsep{0pt}%
    \put(0,0){\includegraphics[width=\unitlength,page=1]{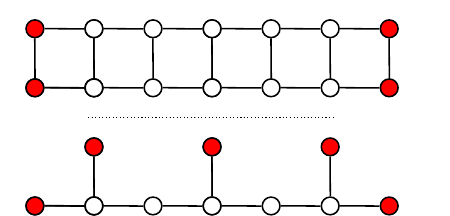}}%
    \put(-0.00307755,0.45074751){\color[rgb]{0,0,0}\makebox(0,0)[lt]{\lineheight{1.25}\smash{\begin{tabular}[t]{l}$(u_0,v_0)$\end{tabular}}}}%
    \put(0.25987157,0.45136804){\color[rgb]{0,0,0}\makebox(0,0)[lt]{\lineheight{1.25}\smash{\begin{tabular}[t]{l}$(u_2,v_0)$\end{tabular}}}}%
    \put(0.52375433,0.45012668){\color[rgb]{0,0,0}\makebox(0,0)[lt]{\lineheight{1.25}\smash{\begin{tabular}[t]{l}$(u_4,v_0)$\end{tabular}}}}%
    \put(0.78534995,0.449506){\color[rgb]{0,0,0}\makebox(0,0)[lt]{\lineheight{1.25}\smash{\begin{tabular}[t]{l}$(u_6,v_0)$\end{tabular}}}}%
  \end{picture}%
\endgroup%
}} 
 \end{center}
 \caption{ A grid graph $\Gamma_{2,7}$ with $\mu(\Gamma_{2,7})=4$ and an induced subgraph $H$ of  $\Gamma_{2,7}$ that is a tree with five leaves. For each graph, vertices in red are  points of a maximum {\mvs}, then $\mu(H)=5>\mu(\Gamma_{2,7})$. } 
\label{fig:noconvex}
\end{figure}
\paragraph{Preliminaries} Let $G=(V,E)$ be a graph and $P\subseteq V$ a set of \emph{points}. Two points are \emph{mutually visible} if there is a shortest path between them with no further point. $P$ is a \emph{\mvs} if its points are pairwise mutually visible.
The \emph{\mvn} of $G$ is the size of any largest \mvs of $G$ and it is denoted $\mu(G)$. By $M(G)$ we denote the set containing all the largest {\mvs}s of $G$. Formally:
$$M(G)=\{P~|~P\subseteq V \mbox{ is a \mvs and } |P|=\mu(G)\}$$


Notice that given a graph $G$ and a set of points $P$, the mutual visibility relation between two points in $P$ is reflexive, symmetric, but not transitive. Then it is different from the visibility relations studied in~\cite{Wu98}, that are all transitive.

Let $H=(V_H,E_H)$ be an induced subgraph of a graph $G$. If $P$ is a \mvs in $G$ then  $P\cap V_H$ is not necessarily a \mvs of $H$. For example, consider a cycle graph $C_n$, $n\geq 4$: it is easy to find a maximum \mvs $P$ of size three. Now consider an induced subgraph $C_n -v$, where $v\not \in P$: it is a path graph. All the points in $P$ are in $C_n -v$, but they are not mutually visible, since one of them is between the other two.  However, the following lemma holds for convex subgraphs of a given graph.
\begin{lemma}\label{lem:mv_subset}
 Let $H=(V_H,E_H)$ be a convex subgraph of $G=(V,E)$. Let $P\subseteq V$ be a \mvs of $G$. Then  $P\cap V_H$ is a \mvs of $H$.
\end{lemma}
\begin{proof}
 Let $u,v$ be two not necessarily distinct vertices of $G$ in $P'=P\cap V_H$, then, by definition of convex subgraph, all the shortest $(u,v)$-paths  in $G$ are in $H$ and one of them is without points in $P$ and then in $P'$. Hence  $u,v$ are mutually visible in $H$. By the generality of $u,v$, $P'$ is a \mvs of $H$.
\end{proof}


Given a graph $G$ and a positive integer $k$, the property $\mu(G)\leq k$ is not a hereditary property for induced subgraphs, i.e., it is possible for an induced subgraph $H$ of $G$ that $\mu(H)>k\geq \mu(G)$. 
Consider the grid graph $\Gamma_{2,7}\sim P_2 \sq P_7$ in Figure~\ref{fig:noconvex}, where  $P_2=(u_0,u_1)$ and $P_7=(v_0,v_1,\ldots,v_6)$. Then, as we will prove in Section~\ref{sec:grid}, $\mu(G)=4$. The induced subgraph $H$ obtained by removing vertices $(u_0,v_0), (u_2,v_0), (u_4,v_0)$, and $(u_6,v_0)$ from $G$ is a tree with five leaves, then, as shown in Figure~\ref{fig:noconvex} and proved in Section~\ref{sec:block}, $\mu(H)=5$. So $\mu(H)>4=\mu(G)$.

However, if we consider convex subgraphs of $G$ the property holds, as stated by the following lemma.   

\begin{lemma}\label{lem:mu_conv}
 Let $H$ be a convex subgraph of a graph $G$. Then  $\mu(H)\leq\mu(G)$.
\end{lemma}
\begin{proof}
 Any \mvs $P$ of $H$ is also a \mvs of $G$, since all the shortest paths between points in $P$ are both in $G$ and in $H$. Then the statement follows.
\end{proof}
 The next two lemmas sets upper bounds to the \mvn of a graph: The following one is based on the {\mvn}s  of certain convex subgraphs.

\begin{lemma}\label{lem:mu_bound}
 Let $G=(V,E)$ be a graph and let $V_1, V_2, \ldots V_k$ be subsets of $V$ such that $\bigcup_{i=1}^k V_i = V$. Then  $\mu(G)\leq\sum_{i=1}^k\mu(\conv{V_i})$.
\end{lemma}
\begin{proof}
Assume $\mu(G) > \sum_{i=1}^k\mu(\conv{V_i})$ and let $P\subseteq V$ be a \mvs such that $|P|=\mu(G)$.  Since $\bigcup_{i=1}^k V_i = V$ any point of $P$ is in at least one  $\conv{V_i}$. Let $P_i$ be the set of vertices that are in $P$ and in $\conv{V_i}$, for each $i=1,2,\ldots,k$. Then $\sum_{i=1}^k |P_i|\geq |P|=\mu(G)> \sum_{i=1}^k\mu(\conv{V_i})$. Hence there exists at least a set $P_j$ such that $|P_j|> \mu(\conv{V_j})$, for some $j$ in $\{1,2,\ldots,k\}$. This is a contradiction since, by Lemma~\ref{lem:mv_subset}, $P_j$ is a \mvs of $\conv{V_j}$ and its size cannot be larger than $\mu(\conv{V_j})$.
\end{proof}

\begin{lemma}\label{lem:upper}
Let $G=(V,E)$ be a graph with $n$ vertices and diameter $d$. Let $c$ the number of vertices of a smallest cycle in $G$, if any. Then
$$\mu(G)\leq \min\{n-d+1, n-c+3\}.$$
\end{lemma}

\begin{proof}
Assume $\mu(G)>n-d+1$ and let $(x_0, x_1,\ldots, x_d)$ be
a diameteral path in $G$. Then this path contains at least three points of any maximum \mvs $P$ of $G$.  Let $i\geq 0$ be the minimum index such that $x_i\in P$ and let $k\leq d$ the maximum index such that $x_k\in P$. Since $x_i$ and $x_k$ must be in mutual visibility there
must exist a shortest path $(x_i=v_0, v_1, \ldots, v_{k-i}=x_k)$ such that vertices $v_1,v_2,\ldots,v_{k-i-1}$ are not in $P$. Then the path
$(x_0,x_1, \ldots, x_i, v_1, \ldots,v_{k-i-1}, x_k, x_{k+1},\ldots, x_d)$ is a diameteral path in $G$ with only two points in $P$. Hence $\mu(G)\leq n-d+1$.

Similarly, assume $\mu(G)>n-c+3$, and let $C_c=(x_0, x_1,\ldots, x_{c-1})$ be a smallest cycle in $G$ with $c$ vertices. Then $C_c$ has at least four points $x_{i_1},x_{i_2},x_{i_3},x_{i_4}$, $i_1<i_2<i_3<i_4$, of any maximum \mvs $P$ of $G$. Since $x_{i_1}$ and $x_{i_3} $ must be in mutual visibility, there
must exist a shortest $(x_{i_1}$,$x_{i_3})$-path without further points in $P$. The cycle given by this path and one of the  $(x_{i_1}$,$x_{i_3})$-paths in $C_c$ has $c$ vertices and at least one point in $P$ less than $C_c$. By repeating the argument, a cycle with $c$ vertices and only three points in $P$ can be found and then $\mu(G)\leq n-c+3$.
\end{proof}

It is worth to notice that for each graph $G$ there exists a \mvs  $P$ such that $|P|=\mu(G)$ and no articulation vertex is in $P$, as shown below.

\begin{lemma}\label{lem:art}
 Let $G=(V,E)$ be  a graph and let $X$ be the set of its articulation vertices. There exists a maximum \mvs $P\in M(G)$ such that $X\cap P=\emptyset$.
\end{lemma}
\begin{proof}
Let $P$ be any \mvs in $M(G)$ and suppose, by contradiction, that there exists a point $x_P\in X\cap P$.
 Let $(V_1,E_1),$ $(V_2,E_2),\ldots, (V_k,E_k), k\geq 2$ be the new connected components of $G-x_P$, created by removing $x_P$. Note that $P\setminus \{x_p\}\subseteq \bigcup_{\ell=1}^k V_\ell$. However, there is only one index $i\in\{1,\ldots,k\}$ such that $P\cap V_i\not = \emptyset$, otherwise there would be two points $u,v$ belonging to two different connected components in $G-x_P$ that are in mutual visibility in $G$. This is impossible since any shortest $(u,v)$-path passes through $x_P$. Then $P'=(P\setminus \{x_P\})\cup \{x'\}$, where $x'\in V_j, j\not=i$, is such that $P'\in M(G)$.
\end{proof}

Before calculating {\mvn}s and maximum {\mvs}s for some graph classes, let us show a first result that compares the \mvn of a graph $G$ with two invariants of $G$.

\begin{lemma}\label{lem:gpDelta}
 Given a graph $G$ with general position number $\gp{G}$ and maximum degree $\Delta(G)$ then $\mu(G)\geq \gp{G}$ and $\mu(G)\geq \Delta(G)$.
\end{lemma}
\begin{proof}
 All vertices of a largest a largest general postion set $S$ of $G$ form a \mvs.  Then $\gp{G}=|S|\leq \mu(G)$.
 
 Let $v$ be a vertex of $G$ with degree $\Delta(G)$, then consider the set $P=N_G(v)$. For any two vertices $x,y\in N_G(v)$ they are adjacent or at distance two in the path $(x,v,y)$. Then, since $v$ is not in $P$, in both cases they are in mutual visibility and then $P$ is a \mvs.
\end{proof}

\begin{remark}~\label{rem:homega}
 For a graph $G$, in~\cite{chand:16} it has been proved that $\gp{G}\geq h(G)$ and $\gp{G}\geq \omega(G)$. Then $h(G)$ and $ \omega(G)$ are also lower bounds for $\mu(G)$.
\end{remark}

The following lemma gives a first taste of the \mvn in two basic graph classes, that will be useful to derive further results.

\begin{lemma}\label{lem:PnCn}
The \mvn of a path graph $P_n$, $n\geq2$, is $\mu(P_n)=2$ and the \mvn of a cycle graph $C_n$, $n\geq 3$, is $\mu(C_n)=3$.
\end{lemma}
\begin{proof}
Since $n\geq 2$ there is an edge $e$ in $P_n$ and the two endpoints of $e$ are mutually visible, so $\mu(P_n)\geq2$. By Lemma~\ref{lem:upper}
$\mu(P_n)\leq2$, since the diameter of $P_n$ is equal to $n-1$. Then $\mu(P_n)=2$
 
 Regarding the cycle graph $C_n=(x_0, x_1\ldots, x_{n-1})$, $\mu(C_n)\geq 3$ since it is always possible to choice  $x_0$, $x_{\lceil \frac n 2\rceil-1}$ and $x_{\lceil \frac n 2\rceil}$ as three points in mutual visibility. By Lemma~\ref{lem:upper}
$\mu(C_n)\leq3$,  Then $\mu(C_n)=3$
\end{proof}

It is interesting to note that $\gp{P_n}=\mu(P_n)=2$ for $n\geq 2$, and $\gp{C_n}=\mu(C_n)=3$ for $n=3$ and $n\geq5$ (see~\cite{manuel:18}). For $n=4$ we have $\gp{C_4}\not =\mu(C_4)$ since $\gp{C_4}=2$ and $\mu(C_4)=3$. Indeed, $C_4$ is the smallest connected graph for which the general position number and the \mvn are different. As we will see in Subsection~\ref{sec:grid}, this difference can be arbitrarily large.


\section{Computational complexity}\label{sec:complexity}
To study the computational complexity of finding a maximum \mvs in a graph, we introduce the following decision problem.
\begin{definition}
\mvp problem: \\
{\sc Instance}: A graph $G=(V,E)$, a positive integer $K\leq |V|$. \\
{\sc Question}: Is there a \mvs $P$ of $G$ such that $|P|\geq K$?
\end{definition}
The problem is hard to solve as shown by the next theorem.
\begin{figure}[t]
   \graphicspath{{fig/}}
   \centering
   \def\svgwidth{\columnwidth}
   \large\scalebox{0.8}{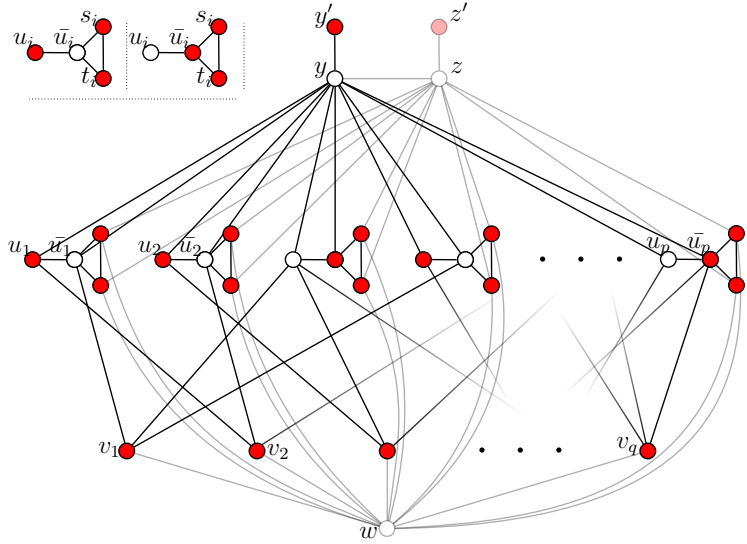}
 \caption{The graph used in Theorem~\ref{theo:NP}. Red vertices are points. Most visible vertices and edges represent the main part of the graphs. The rest is added to ensure the mutual visibility among points. Top left: The true-setting gadget used to represent a variable $x_i$ with two maximum {\mvs}s representing the two possible truth assignments: $x_i$ is false if and only if $u_i$ is a point.}
\label{fig:NP}
\end{figure} 

\begin{theorem}\label{theo:NP}
 \mvp  is NP-complete.   
\end{theorem}  

\begin{proof} 
 Given  a set of points $P\subseteq V$ of $G$, we can test in polynomial time whether $P$ is a \mvs or not (see also Algorithm \mva). Consequently, the problem is in NP. 

 We will now prove that the \sat problem, shown as NP-complete in~\cite{Karp72}, polynomially reduces to \mvp. 
 \begin{quote}  
 A \sat instance $\Phi$ is defined as a set $X=\{x_1,x_2,\ldots,x_p\}$ of $p$ boolean variables and a set $C$ of $q$ clauses, each  defined as a set of three literals: every variable $x_i$ corresponds to two literals $x_i$ (the positive form) and $\bar {x_i}$ (the negative form). To simplify the notations we will denote by $\{\ell_1, \ell_2, \ell_3\}$ the clause with literals $\ell_i, i=1,2,3$, without distinction between the orders in which they are listed. A truth assignment assigns a Boolean value ($True$ or $False$) to each variable, corresponding to a truth assignment of opposite values for the two literals $x_i$ and $\bar {x_i}$: $\bar {x_i}$ is $True$ if and only if $x_i$ is $False$. 
 A clause is satisfied if at least one of its literals is satisfied. The \sat problem asks whether there is a truth assignment satisfying all clauses.
\end{quote}
In what follows,  we assume there are at least three clauses such that their (pairwise) intersection is empty. Any instance $\Phi$ that does not satisfy this constraint can be transformed into an instance $\Phi'$ with such three clauses by adding five new variables $a,b,c,d,e$ and the required three clauses $\{a, \bar a, b\}$, $\{\bar b, c, \bar c\}$, $\{d,\bar d, e\}$ that are  always satisfied for each truth assignment. Then the \sat instance $\Phi$ has a Yes answer if and only if $\Phi'$ has a Yes answer.

We transform \sat to \mvp. Let $X = \{x_1,x_2,\ldots,x_p\}$
and $C = \{c_1,c_2,\ldots,c_q\}$ be any instance of \sat. We must construct a
graph $G = (V,E)$ and a positive integer $K \leq |V|$ such that $G$ has a \mvs  of size $K$ or more if and only if $C$ is satisfiable.

For each variable $x_i\in X$, there is a true-setting convex subgraph of $G$ $T_i=(V_i,E_i)$, with $V_i=\{u_i,\bar {u_i}, s_i, t_i\}$ and $E_i=\{u_i\bar {u_i},\bar{u_i}s_i,\bar{u_i}s_i, s_{i}t_i\}$. See the top left part of Figure~\ref{fig:NP} for a drawing of $T_i$ and the two possible maximum {\mvs}s.  Notice that each of the two maximum {\mvs}s of $T_i$ contains either $u_i$ or $\bar {u_i}$.

For each clause $c_j \in C$, there is a vertex $v_j$ and, for each literal $x_i$ (or $\bar{x_i}$) in $c_j$ there is in an edge $v_ju_i$ (or an edge $v_j \bar{u_i}$, respectively). Moreover, there is a vertex $w$ and edges ${v_j}w$ for each $j=1,2,\ldots, q$.

There are four more vertices in $V$, that is $y,y',z,z'$. For each $i\in\{1,2,\ldots, p\}$ there are edges ${u_i}y$, $\bar{u_i}y$, ${s_i}z$, ${t_i}z$, ${s_i}w$, ${t_i}w$. Finally, $E$ contains edges $yz$, $yy'$ and $zz'$. 

A representation of $G$ is given in Figure~\ref{fig:NP}.

The construction of our instance of \mvp is completed by setting $K= 3p+q+2$.
It is easy to see how the construction can be accomplished in polynomial time. All that remains to be shown is that $C$ is satisfiable if and only if
$G$ has a \mvs of size $K$ or more.

First, suppose that $t: X\rightarrow \{True,False\}$ is a satisfying truth assignment for $C$. The corresponding set of points $P$ includes vertices $u_i$ if $t(x_i)$ is $False$, and  $\bar{u_i}$ otherwise, for each $i\in\{1,2,\ldots, p\}$. Moreover $y'$, $z'$, $v_j$, $s_i$, $t_i$ are in $P$, for each possible value of $i$ and $j$. No further vertex is in $P$. Then $|P| =  3p+q+2=K$. It remains to show that $P$ is a \mvs. Clearly, $y'$ is in mutual visibility with $z'$. Let $ST=\left\{s_i, t_i~|~i\in\{1,2,\ldots, p\}\right\}$,  $U=\left\{u_i, \bar{u_i}~|~i\in\{1,2,\ldots, p\}\right\}$, $D=\left\{v_j~|~j\in\{1,2,\ldots, q\}\right\}$. Each vertex in $ST$ is in mutual visibility with all the points in its true-setting subgraph and with all the other points in $P$ thanks to shortest paths passing through vertices $w,y$, and $z$ that are not in $P$ (e.g., for $t_i\not \in T_1$, the paths $(t_i,z,s_1), (t_i,z,t_1), (t_i,z,z'), (t_i,z,y,y'), (t_i,z,y,u_1), (t_i,z,y,\bar{u_1}), (t_i,w,v_1)$).  

All the points in $D$ are in mutual visibility through shortest paths of length two via vertex $w$. More interesting is to show that each point $v\in D$ is in mutual visibility with $y'$ (and with $z'$). Point $v$ corresponds to a clause $c\in C$ and, since $C$ is satisfiable, there is a vertex  $u$ in $N_G(v)\cap U$ that is not in $P$ corresponding to a $True$ literal in $c$.
Then the shortest paths $(v,u,y,y')$ and $(v,u,y,z,z')$ show that $v$ is in mutual visibility with $y'$ and $z'$. Finally, each point in $U$ is in mutual visibility with points $y'$, $z'$, and all the other points in $U$, because of shortest paths passing through $y$ and $z$.
Regarding the mutual visibility of points in $U$ with points in $D$, let $v_j$ be a point in $D$ corresponding to a clause $c_j$ and let $x_i$ (or $\bar{x_i})$ be a literal in $c_j$ corresponding to vertex $u_i$ (or $\bar{u_i}$). If $t(x_i)$ is True then either the point $\bar{u_i}$ is connected to $v_j$ with the path $(\bar{u_i},u_i,v_j)$ or $\bar{u_i}$ is adjacent to $v_j$. Otherwise, if $t(x_i)$ is False then either the point $\bar{u_i}$ is adjacent to $v_j$ or connected to $v_j$ via $(u_i,\bar{u_i},v_j)$. Similarly for all the literals in $c_j$. If $x_i$ is not in $c_j$ then  point $u_i$ (or $\bar{u_i}$) is in mutual visibility with $v_j\in D$ thanks to a shortest path $(u_i,y,u',v_j)$ (or $(\bar{u_i},y,u',v_j)$), where the vertex $u'\in D$ is in correspondence with a $True$ literal in $c_j$. This concludes the first part of the proof.

Conversely, let us suppose that there is a set $P\subseteq V$ of points such that $|P|\geq K=3p+q+2$. In $C$ there are three clauses that do not share any variable. Assume, without loss of generality, that these three clauses are, $c_1$, $c_2$, $c_3$. Then the star subgraph $H$ of $G$ induced by vertices $v_1$, $v_2$, $v_3$ and $w$ is a convex subgraph of $G$. Convex subgraph of $G$ are $T_i$, the path graph $H'=(y',y,z,z')$ and each subgraph $L_j\sim K_1$ consisting in a single vertex $v_j$, $j=4,\ldots q$. The union of the vertices of these convex subgraphs is $V$ then, by applying Lemmas~\ref{lem:mu_bound}, we have:
$$\mu(G) \leq \mu(H)+\mu(H')+\sum_{i=1}^p \mu(T_i)+\sum_{j=4}^q \mu(L_j) = 3+2+3p+(q-3)=3p+q+2$$

The above inequality holds since it is not difficult to see that $\mu(H)=3$ (see also Corollary~\ref{cor:tree}), $\mu(H')=2$ by Lemma~\ref{lem:PnCn}, and, by enumeration, that $\mu(T_i)=3$. The \mvn $\mu(G)$ is the size of a largest \mvs in G, then $|P|= K=3p+q+2$. Since $y$ and $z$ are articulation vertices we can assume, by Lemma~\ref{lem:art}, they are not in $P$.

Moreover, at least one vertex for each $T_i$ is not in $P$: call $Q$ the set of such vertices. Then the points in $P$ are a subset of $V'=V\setminus (Q \cup\{y,z\})$. Since $|V|=4p+q+5$ and $|Q|\geq p$, then $|V'|\leq 3p + q+3$. Hence at most one vertex in $V'$ is not in $P$. Consequently, at least two vertices among $v_1$, $v_2$, $v_3$ of $H$ are in $P$ (say $v_1$, $v_2$), and since $H$ is a convex subgraph of $G$ then the only shortest path between $v_1$ and $v_2$ is $(v_1,w,v_2)$. This implies that $w$ is not in $P$, otherwise $v_1$ and $v_2$ are not in mutual visibility. 

In conclusion, all the vertices in $D$ are in $P$, $y'$ and $z'$ are in $P$ and three vertices for each $T_i$ are in $P$, and in particular exactly one vertex among $u_i$ and $\bar{u_i}$ is in $P$. Now, consider a point $v_j$ in $D$ and its corresponding clause $c_j$. Since $v_j$ and $y'$ are mutually visible, at least one vertex in $N_G(v_j) \cap U$ is not in $P$ an then, the corresponding literal is $True$ and $c_j$ is satisfied. By the generality of $c_j$ all the clauses are satisfied. 
\end{proof}
 
Theorem~\ref{theo:NP} shows that \mvp is hard, however the following problem, which asks to test if a given set of points is a \mvs, can be solved in polynomial time.

\begin{definition}
\mvt: \\
{\sc Instance}: A graph $G=(V,E)$ and $P\subseteq V$. \\
{\sc Question}: Is $P$ a  \mvs of $G$?
\end{definition}

The solution is provided by means of Algorithm \mva that in turn uses Procedure \bfs as a sub-routine.  Procedure \bfs and Algorithm \mva are shown in Figures~\ref{alg:bfs} and~\ref{alg:algoMV}, respectively.

\begin{algorithm}[ht]
\SetKwInput{Proc}{Procedure}
\Proc{\bfs}
\SetKwInOut{Input}{Input}
\Input{A connected graph $G=(V,E)$, a set of points $P$, $v\in V$, a boolean $t$ }
\SetKwInOut{Output}{Output}
\Output{The distance vector of $v$ from any vertex $u$ in $P$ calculated in $G$, if $t$ is True, otherwise calculated in $G- P\setminus\{u,v\}$.
}
\BlankLine
\BlankLine
$D[u]:=\infty ~\forall u\in V$\;\label{line:startin}
$DP[p]:=\infty ~\forall p\in P$\;
$D[v]:=0$\;\label{line:endin}
\lIf{$v\in P$}{$DP[v]:=0$}
Let $Q$ be a queue\;
$Q.enqueue(v)$\;
\While{$Q$ is not empty and $\exists p\in P, D[p]=\infty$ \label{line:ciclo}}
   { $u := Q.dequeue()$\; \label{line:deq}
     \For {each $w$ in $N_G(u)$}{
        \If{$D[w]=\infty$}{
                   $D[w]:=D[u]+1$\;\label{line:up}
                   \lIf{ $w \in P$}{
                       $DP[w]:=D[w]$} 
                   \lIf {$t$ or $w \not \in P$}{ 
                       $Q.enqueue(w)$\label{line:enq}}
                   }    
        }
     }                  
                       
\Return DP
 
\caption{Procedure \bfs 
}
\label{alg:bfs}
\end{algorithm}

\begin{algorithm}[t]
\SetKwInput{Proc}{Algorithm}
\Proc{\mva}
\SetKwInOut{Input}{Input}
\Input{ A graph $G=(V,E)$ and a set of points $P\subseteq V$}
\SetKwInOut{Output}{Output}
\Output{ True if $P$ is a \mvs, False otherwise }
\BlankLine
\If {points in $P$ are in different connected components of $G$\label{line:comp}}
    {\Return False}
Let $H$  be the connected component of $G$ with points\;
\For{\mbox{each} $p \in P$\label{line:loop}}{
       \If {\bfs(H,P,p,False) $\not =$ \bfs(H,P,p,True)\label{line:call}}
          {\Return False}\label{line:exit}
          }
\Return True\label{line:end}
\caption{Algorithm \mva }
\label{alg:algoMV}
\end{algorithm}

\begin{theorem}\label{theo:P}
 Algorithm \mva solves \mvt in $O(|P|(|V|+|E|))$ time.  
\end{theorem}  
\begin{proof}
When $G$ is connected, Algorithm \mva (see Figure~\ref{alg:algoMV}) calculates the distances between any pair of points $u,v\in P$ both in $G$ and in $G- P\setminus\{u,v\}$, the graph obtained by removing all the points in $P$ except $u$ and $v$ (loop at Line~\ref{line:loop}). To this end, \mva uses Procedure \bfs (see. Figure~\ref{alg:bfs}). If the distances are equal (and Line~\ref{line:end} is reached) then there exits a shortest $(u,v)$-path   without points, that is $u$ and $v$ are in mutual visibility. Otherwise, $P$ is not a \mvs (Line~\ref{line:exit}).

Procedure \bfs is a variant of the breadth-first search algorithm that updates two distance vectors: the distance vector $D$, for the distances of $v$ with each vertex in the graph, and the distance vector $DP$, for the distances of $v$ with the points in $P$. These distance vectors are initialized at Lines~\ref{line:startin}--\ref{line:endin}. To track all the visited vertices, a queue $Q$ is initialized with the vertex $v$. Then, at Line~\ref{line:ciclo}, a loop starts, ending when all the points are visited or there are no more vertices to visit. Within the loop, the first vertex $u$ in $Q$ is dequeued at Line~\ref{line:deq}. For each non-visited neighbor $w$ of $u$   its distance $D[w]$ from $v$ is correctly updated to $D[u]+1$ (see Line~\ref{line:up}). If $w$ is a point, this distance is recorded in $DP$. Finally, at Line~\ref{line:enq}, $w$ is enqueued in $Q$ if it is not a point or the distances must be calculated in $G$ (that is, if $t$ is True). Note that if $t$ is False and $w$ is a point, $w$ is not enqueued in $Q$ because any shortest path between $v$ and a point $p\in P$, $p\not = w$, useful to calculate $d_G(v,p)$ and to test the mutual visibility of $v$ and $p$, cannot pass through $w$.
Procedure \bfs ends by returning the distance vector $DP$. 

Algorithm \mva first checks if the points in $P$ are in different connected components of $G$ at Line~\ref{line:comp}. In this case $P$ is not a \mvs and the algorithm correctly returns False.
If all the points are in the same connected component $H=(V_H,E_H)$ (and hence in $G$, if it is connected), for each point $p$ in $P$ it calculates the distances of $p$ from each other point $u\in P$, both in $H$ and in $H-P\setminus \{p,u\}$ (see Line~\ref{line:call}). If at least one of these distances is different in the two graphs, then Algorithm \mva correctly returns False, otherwise True.

Procedure \bfs works in  $O( | V | + | E | ) $ time, since every vertex and every edge will be explored in the worst case. Algorithm \mva calls Procedure \bfs at most two times for each $p\in P$, then the overall time is $O(|P|( | V | + | E |) ) $.
\end{proof}



\section{\Mvs for special graph classes}\label{sec:graphs}
In this section we study the \mvn for specific graph classes and provide some  results useful to calculate maximum {\mvs}s in polynomial time in these graphs.

\subsection{Graph characterization by \mvn}
The following lemma characterizes some graph classes in terms of their \mvn.
\begin{lemma}\label{lem:char}
Let $G=(V,E)$ be a graph such that $|V|=n$. Then  
 \begin{enumerate}
  \item $\mu(G)=1 \iff G \sim \overline{K_n}$ (if $G$ is connected: $\mu(G)=1 \iff G \sim K_1$);
  \item $\mu(G)=2\iff$ $n>1$ and $G\sim P_n$ or $G$ is the disjoint union of at most $n-1$ path graphs;
  \item $\mu(G)=|V|\iff G \sim K_n$;
  \item if $G$ is connected and $|V|>2$: $\mu(G)=|E|\iff G \sim K_{1,n-1}$ or $G$ is a triangle
 \end{enumerate}

\end{lemma}
\begin{proof}
 \begin{enumerate}
  \item $(\Rightarrow)$ Since $\mu(G)=1$, $E$ must be empty otherwise there exist  $xy\in E$, and $x$ and $y$ are mutually visible, so $\mu(G)\geq 2$. Hence $G$ is a graph with $n$ vertices and no edges, that is $\overline{K_n}$\\
        $(\Leftarrow)$ Since $E$ is empty there are no paths between vertices, then any \mvs cannot have more than one point. Hence $\mu(G)=1$.
  \item $(\Rightarrow)$ Since $\mu(G)=2$, then $n>1$. Assume now by contradiction that $G$ is not isomorphic to $P_n$ and $G$ is not the disjoint union of $n-1$ path graphs. Since at least one connected component of $G$ has at least two vertices, $\Delta(G)$ cannot be zero. If $\Delta(G)=1$, $G$ would be the disjoint union of $K_1$ and $P_2$ graphs, but this is not possible. For the same reason, if $\Delta(G)=2$ at least one connected component must be a cycle graph $C_k$, for a certain $k$, impossible since $\mu(C_k)=3$ by Lemma~\ref{lem:PnCn}. Then $\Delta(G)\geq 3$, but by Lemma~\ref{lem:gpDelta}, also in this case $\mu(G)\geq 3$.\\
  $(\Leftarrow)$ Since $n\geq2$, $\mu(G)=2$ by Lemma~\ref{lem:PnCn} applied to a connected component of $G$ with more than two vertices, that must exist since the connected components are at most $n-1$.
  \item $(\Rightarrow)$ Since all vertices are in mutual visibility then $G$ is connected. Moreover each $u,v\in V$ must be adjacent otherwise in any shortest path connecting them there is at least one vertex, that is a point since the \mvs is $V$.\\
  $(\Leftarrow)$ Obvious since all pair of vertices are adjacent and then mutually visible.
 \item $(\Rightarrow)$ Since $G$ is connected, $|E|\geq |V|-1$. Moreover $|V|\geq \mu(G)= |E|$. Then $|E|\leq |V|\leq |E|+1$. If $|V|=|E|=\mu(G)$ then, by point 3), $G$ is a clique graph and then a triangle (the only case where $|E|=\frac{n(n-1)} 2 = n = |V|$). If $|V|=|E|+1$, $G$ is a tree with $|V|>2$ vertices and $\mu(G)=|V|-1$ points. Then only one vertex is not a point. It can be a leaf only if $|V|=3$ (and then $G\sim K_{1,2}$) otherwise there is a convex path connecting two leaves with at least three points in mutual visibility, a contradiction by Lemmas~\ref{lem:mv_subset} and~\ref{lem:PnCn}. For the same reason, when $G\not \sim K_{1,2}$, it is the only vertex that is not a leaf. Then the graph $G$ is a star with $n$ vertices, that is $G\sim K_{1,n-1}$.\\
 $(\Leftarrow)$ Obvious if $G$ is a triangle, otherwise, by Lemma~\ref{rem:homega}, $\mu(G)\geq\Delta(G)=deg(v)=|V|-1=|E|$, where $v$ is the center of $G$. However $\mu(G)$ cannot be larger than $|E|=|V|-1$ otherwise  $v$ should be a point, preventing the mutual visibility among the pendant vertices.
\end{enumerate}
\end{proof}

\subsection{Block graphs, Trees and Geodesic graphs}\label{sec:block}
\begin{figure}[t]
   \graphicspath{{fig/}}
   \centering
   \def\svgwidth{\columnwidth}
   \large\scalebox{0.5}{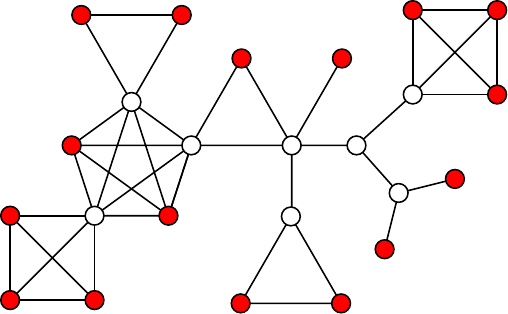}
 \caption{A block graph. Red vertices are points of the maximum \mvs.}
\label{fig:block}    
\end{figure}  

A \emph{block graph} is a graph in which every maximal biconnected subgraph (called \emph{block}) is a clique (see Figure~\ref{fig:block}). 
As an application of Lemma~\ref{lem:art} on articulation vertices, the next Theorem characterizes the largest {\mvs}s (an then the \mvn) for block graphs.

\begin{theorem}\label{theo:block}
 Let $G=(V,E)$ be a connected block graph and $X$ the set of its articulation vertices. $V\setminus X$ is a \mvs of $G$ and $\mu(G)=|V\setminus X|$.
\end{theorem}
\begin{proof} 
 By Lemma~\ref{lem:art} there exists a maximum \mvs $P\in M(G)$ without vertices in $X$. To show that $P$ includes all the vertices of $G$  in $V\setminus X$, consider two of them $u,v$ and the shortest $(u,v)$-path (note that in block graphs the shortest path between two vertices is unique). If $u$ and $v$ belong to the same block then they are adjacent since a block is a clique by definition. Otherwise, the shortest $(u,v)$-path passes only through articulation vertices of $G$, since they are induced paths. Then $u$ and $v$ are mutually visible. By the generality of $u$ and $v$ the lemma holds. 
\end{proof}

An immediate consequence of Theorem~\ref{theo:block} is the following corollary holding for trees.

\begin{corollary}\label{cor:tree}
 Let $T=(V,E)$ be a tree and $L$ the set of its leaves. Then $L$ is a \mvs and $\mu(T)=|L|$.
\end{corollary}
\begin{proof}
 A tree is a block graph where the blocks are the edges ($K_2$ subgraphs) in $E$ and each vertex in $V$ that is not a leaf is an articulation vertex. Then $L$ is a maximum \mvs by Theorem~\ref{theo:block}.
\end{proof}

Figure~\ref{fig:noconvex} shows the maximum \mvs $P$ of a tree. However,  the maximum \mvs of trees and block graphs is not unique.
This is the case when the removal of an articulation point creates a new component that is a path $P_n$, $n\geq 2$. In that case, we can create several {\mvs}s of maximum size by choosing single vertices of $P_n$ for each of them.

In~\cite{manuel:18} the same result on block graphs is achieved for general position sets by using the concept of simplicial vertex. A vertex is \emph{simplicial} if its neighbours induce a complete graph. 

\begin{theorem}
 \emph{(\cite{manuel:18}, Th. 3.6)} Let $S$ be the set of simplicial vertices of a block graph $G$. Then $S$ is a maximum general position set and hence $\gp{G}=|S|$. 
\end{theorem}

Indeed, in Figure~\ref{fig:block}, red vertices are simplicial vertices. Then, for block graphs, since the vertices are simplicial vertices or articulation vertices, we have $\gp{G}=\mu(G)$. This result can be easily generalized to \emph{geodetic graphs}: a graph is geodetic if the shortest path between any pair of vertices is unique, like in block graphs and trees. 

\begin{remark}\label{rem:geo} If $G$ is a a geodetic graph then $\gp{G}=\mu(G)$.
\end{remark}
The contrary of Remark~\ref{rem:geo} is not true. Cycle graphs $C_{2n}$, are not geodetic, but, for $n\geq3$, $\gp{C_{2n}}=\mu(C_{2n})=3$  (see~\cite{manuel:18} and Lemma~\ref{lem:PnCn}).

\subsection{Grids, Tori}\label{sec:grid}

\begin{figure}[t]
   \graphicspath{{fig/}}
   \centering
   \def\svgwidth{\columnwidth}
   \large\scalebox{0.8}{\input{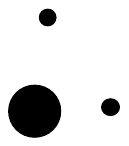}}
 \caption{On the left: mutual-visibility sets for grid graphs $\Gamma_{1,1}$, $\Gamma_{1,2}$, $\Gamma_{2,2}$, and $\Gamma_{2,3}$. On the right: two non isomorphic mutual-visibility sets for $\Gamma_{3,3}$. The size of each \mvs determines the \mvn of the corresponding graph.}
\label{fig:smallgrids}
\end{figure} 

\begin{figure}[t]
   \graphicspath{{fig/}}
   \centering
   \def\svgwidth{\columnwidth}   
   \large\scalebox{0.20}{{\huge 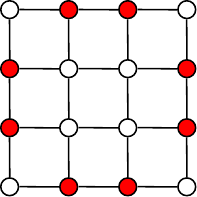}}~a)~~~~~~~~~~~~
   \large\scalebox{0.50}{{\huge 
\begingroup%
  \makeatletter%
  \providecommand\color[2][]{%
    \errmessage{(Inkscape) Color is used for the text in Inkscape, but the package 'color.sty' is not loaded}%
    \renewcommand\color[2][]{}%
  }%
  \providecommand\transparent[1]{%
    \errmessage{(Inkscape) Transparency is used (non-zero) for the text in Inkscape, but the package 'transparent.sty' is not loaded}%
    \renewcommand\transparent[1]{}%
  }%
  \providecommand\rotatebox[2]{#2}%
  \newcommand*\fsize{\dimexpr\f@size pt\relax}%
  \newcommand*\lineheight[1]{\fontsize{\fsize}{#1\fsize}\selectfont}%
  \ifx\svgwidth\undefined%
    \setlength{\unitlength}{203.47987124bp}%
    \ifx\svgscale\undefined%
      \relax%
    \else%
      \setlength{\unitlength}{\unitlength * \real{\svgscale}}%
    \fi%
  \else%
    \setlength{\unitlength}{\svgwidth}%
  \fi%
  \global\let\svgwidth\undefined%
  \global\let\svgscale\undefined%
  \makeatother%
  \begin{picture}(1,0.88132548)%
    \lineheight{1}%
    \setlength\tabcolsep{0pt}%
    \put(0,0){\includegraphics[width=\unitlength,page=1]{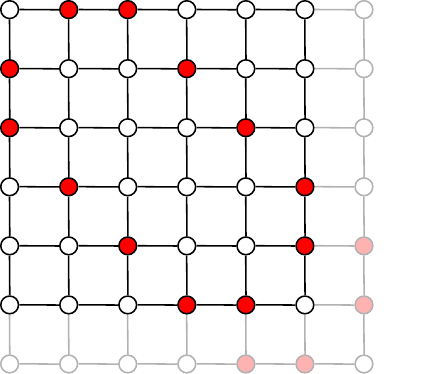}}%
    \put(0.59806658,0.18274323){\color[rgb]{0,0,0}\makebox(0,0)[lt]{\lineheight{1.25}\smash{\begin{tabular}[t]{l}$u$\end{tabular}}}}%
    \put(0.73758459,0.32155154){\color[rgb]{0,0,0}\makebox(0,0)[lt]{\lineheight{1.25}\smash{\begin{tabular}[t]{l}$v$\end{tabular}}}}%
    \put(0.87684077,0.32258103){\color[rgb]{0,0,0}\makebox(0,0)[lt]{\lineheight{1.25}\smash{\begin{tabular}[t]{l}$z$\end{tabular}}}}%
    \put(0.87741129,0.1820726){\color[rgb]{0,0,0}\makebox(0,0)[lt]{\lineheight{1.25}\smash{\begin{tabular}[t]{l}$y$\end{tabular}}}}%
    \put(0.73758459,0.04255343){\color[rgb]{0,0,0}\makebox(0,0)[lt]{\lineheight{1.25}\smash{\begin{tabular}[t]{l}$x$\end{tabular}}}}%
    \put(0.59686121,0.04343141){\color[rgb]{0,0,0}\makebox(0,0)[lt]{\lineheight{1.25}\smash{\begin{tabular}[t]{l}$w$\end{tabular}}}}%
  \end{picture}%
\endgroup%
}}~b)
 \caption{a) The unique maximum mutual-visibility set for $\Gamma_{4,4}$. b) A maximum mutual-visibility set for $\Gamma_{6,6}$. An extension of this set for $\Gamma_{7,7}$, when $\Gamma_{6,6}$ is seen as one of its subgraphs, is obtained by removing points $u$ and $v$ and by adding points $w$, $x$, $y$, and $z$.}
\label{fig:grids}
\end{figure} 

For grid graphs $\Gamma_{m,n}$,   Figure~\ref{fig:smallgrids} represents the {\mvs}s of maximum size for small values of $m$ and $n$, and Theorem~\ref{theo:grid} gives the values of $\mu(\Gamma_{m,n})$ for $m>3$ and $n>3$.
These values are based on maximum {\mvs}s shown in Figure~\ref{fig:grids}. Furthermore, Table~\ref{tab:grid} shows the values of $\mu(\Gamma_{m,n})$ for all the possible settings of $m$ and $n$, $m\leq n$.

\begin{table}[t]
\begin{center}
\begin{tabular}{c|c|c|c|c}
$m$&$n$&Graph $G$&$\mu(G)$&Reference\\\hline
1 &1 & $K_1$& 1&Lemma~\ref{lem:char}\\\hline
1 &$n>1$& $P_n$&2&Lemma~\ref{lem:PnCn}\\\hline
2 & 2 & $C_4$ & 3&Lemma~\ref{lem:PnCn}\\\hline
2 & $n>2$& $\Gamma_{2,n}$&4&Lemmas~\ref{lem:mu_bound} and~\ref{lem:PnCn}\\\hline
3 & 3 &  $\Gamma_{3,3}$&5&Figure~\ref{fig:smallgrids}\\\hline
3 & $n>3$ &  $\Gamma_{3,n}$&6&Lemmas~\ref{lem:mu_bound} and~\ref{lem:PnCn}\\\hline
4 & 4 &  $\Gamma_{4,4}$&8&Figure~\ref{fig:grids} \\\hline
$m> 3$&$n> 3$&  $\Gamma_{m,n}$&$2m$&Theorem~\ref{theo:grid}\\
\end{tabular}
\end{center}
 \caption{Values of $\mu(G)$ when $G\sim\Gamma_{m,n}$, for all the possible value of $m$ and $n$ such that $m\leq n$.}\label{tab:grid}
\end{table}

\begin{theorem}\label{theo:grid}
 Let $\Gamma_{m,n}=P_m\sq P_n$ be a grid graph such that $m>3$ and $n>3$ then
 $$\mu(\Gamma_{m,n})=2 \cdot \min(m,n).$$
\end{theorem}
\begin{proof}
 Let $P_m =(u_0, u_1,\ldots, u_{m-1})$ and $P_n =(v_0, v_1,\ldots, v_{n-1})$.
 In each subgraph $((u_0, v_i), (u_1,v_i),\ldots, (u_{m-1},v_i))$, representing  the $i$-row of $\Gamma_{m,n}$, there are at most two points as an immediate consequence of Lemmas~\ref{lem:mu_bound} and~\ref{lem:PnCn}, since a row is a convex subgraph of $\Gamma_{m,n}$ and is a path. The same holds for each subgraph $((u_j, v_0), (u_j,v_1),\ldots, (u_j,v_{n-1}))$, representing  the $j$-column of the grid. Then $\mu(\Gamma_{m,n})\leq  2 \cdot \min(m,n)$. To show that the equality holds, let $k=\min(m,n)$ and consider a subgraph $\Gamma_{k,k}$ of  $\Gamma_{m,n}$. If $k=4$, the unique maximum \mvs of $\Gamma_{4,4}$ is given by:
 $(u_1,v_0), (u_2,v_0), (u_0,v_1), (u_3,v_1),(u_0, v_{2}), (u_3, v_2), (u_1,v_3), (u_2,v_3)$ and is represented in Figure~\ref{fig:grids}a. Then $\mu(\Gamma_{k,k})=8$, and since there are two points for each row and each column, then $\mu(\Gamma_{m,n})=8=2 \cdot \min(m,n)$.
 
For $k\geq 5$, consider again a grid subgraph $\Gamma_{k,k}$ of $\Gamma_{m,n}$ and the set of points:
 $(u_1,v_0), (u_2,v_0), (u_0,v_1), (u_3,v_1),$ $
 (u_{j-2},v_j),(u_{j+2},v_j)$, for each $j = 2, \ldots, k-3$, and
 $ (u_{k-4}, v_{k-2}), (u_{k-1}, v_{k-2}), (u_{k-3},v_{k-1}), (u_{k-2},v_{k-1})$.
 This set generalizes the solution given for $k=4$ and is represented in Figure~\ref{fig:grids}b. Since there are two points for each row and each column and all the points are in mutual visibility, then $\mu(\Gamma_{m,n})=2 \cdot \min(m,n)$.
\end{proof}

In~\cite{manuel:18b} it is proved that $\gp{\Gamma_{n,n}}=4$ for $n\geq3$. Since $\mu(\Gamma_{n,n})=2n$ for $n>3$, then the difference between the two numbers can be arbitrarily large.

\begin{figure}[t]
   \graphicspath{{fig/}}
   \centering
   \def\svgwidth{\columnwidth}   
   \large\scalebox{0.25}{{\huge 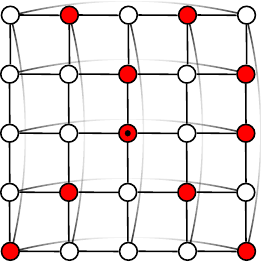}}~a)~~~~~~~~~~~~~~~
   \large\scalebox{1}{{\huge 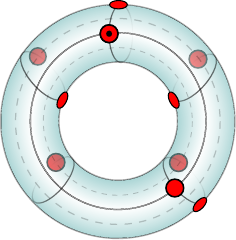}}~b)\\~\\
   \large\scalebox{0.35}{{\huge 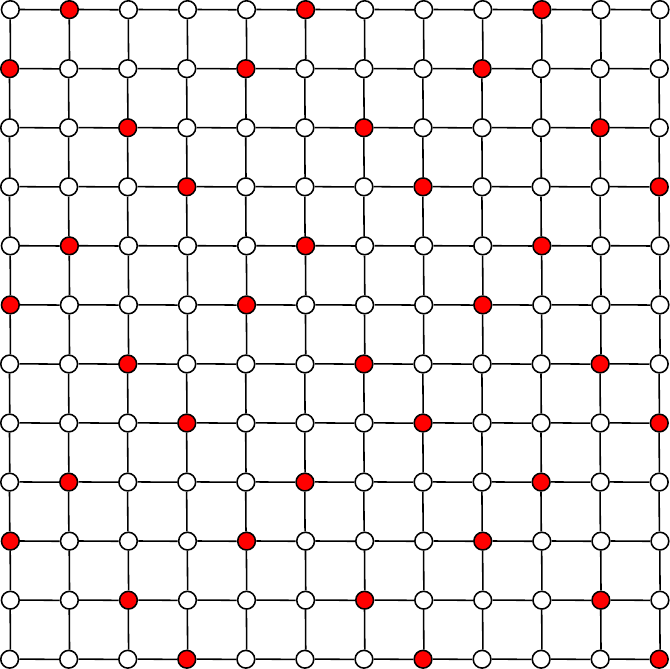}}~c)~~~~~~~~~~~~~~
   \large\scalebox{0.28}{{\huge 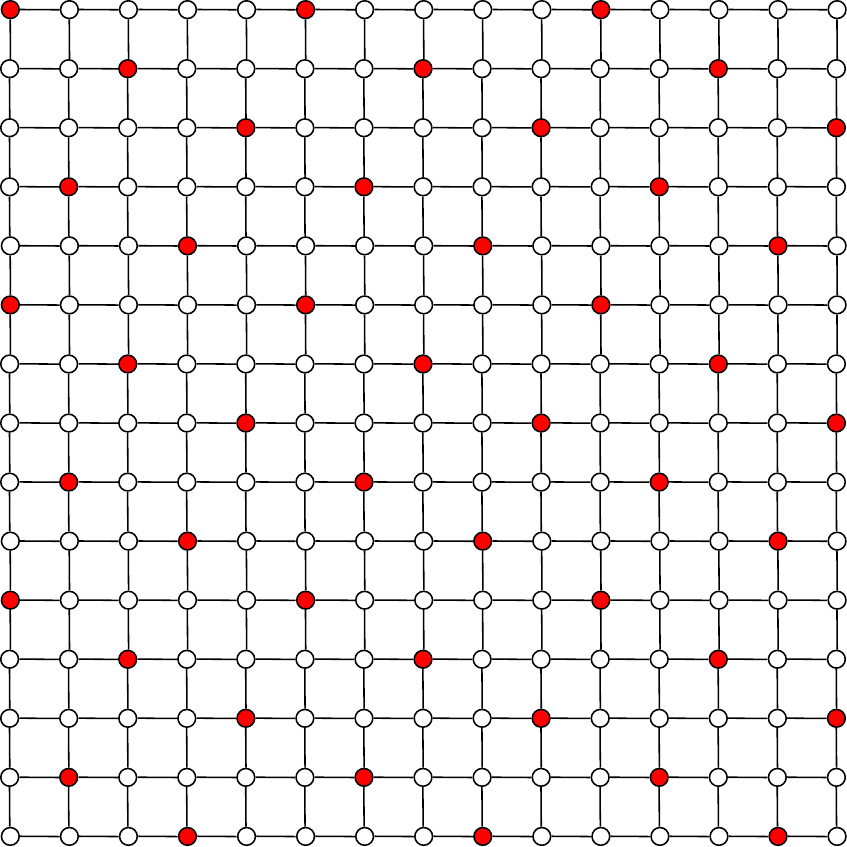}}~d)
 \caption{a) A torus $C_5 \sq C_5$. Vertices in red form a maximum \mvs.  b) The same graph represented as a three-dimentional torus. The dotted vertex corresponds to the dotted vertex in a). c) A solution for a torus $T_{12,12}$ such that $\mu(T_{12,12})=3\cdot12$. d) A solution for a torus  $T_{15,15}$ such that $\mu(T_{15,15})=3\cdot15$. }\label{fig:tori}
\end{figure} 

\bigskip
For tori $T_{m,n}=C_m\sq C_n$, notice that each copy of $C_m$ and $C_n$ is a convex subgraph of $T_{m,n}$. Then, by Lemmas~\ref{lem:mu_bound} and~\ref{lem:PnCn}, we derive:

\begin{corollary}\label{cor:tori}
  Let $T_{m,n}=C_m\sq C_n$ be a torus such that $m\geq3$ and $n\geq3$ then
 $$\mu(T_{m,n})\leq 3 \cdot \min(m,n).$$
\end{corollary}

However, the problem of finding $m$ and $n$ such that the \mvn of $T_{m,n}$ is equal to $3 \cdot \min(m,n)$ is still open. In general, solutions for tori are quite irregular, like that shown in Figure~\ref{fig:tori}a for a torus $T_{5,5}$, where the upper bound is not reached. It would be interesting to find the values of $m$ such that  $\mu(T_{m,m})$ reaches the upper bound of Corollary~\ref{cor:tori}. There are no {\mvs}s for tori $T_{m,m}$ such that $\mu(T_{m,m})=3\cdot m$, for $m\leq 11$ (result obtained with a backtracking algorithm that explored a space of $\binom{m}{3}^m$ possible solutions using Algorithm \mva as subprocedure). As shown in Figures~\ref{fig:tori}c and~\ref{fig:tori}d, for $m=12$ and $m=15$ there are tori such that $\mu(T_{12,12})= 3 \cdot 12$ and $\mu(T_{15,15})= 3 \cdot 15$. These solutions were found without the help of a computer, but a scalable solution, such as the one provided in Theorem~\ref{theo:grid} and shown in Figure~\ref{fig:grids} for grids, is not available. With respect to grids, the main difficulty is due to the fact that, for $m'>m$ and $n'>n$, $T_{m,n}$ is not a subgraph of $T_{m',n'}$, whereas $\Gamma_{m,n}$ is a subgraph of $\Gamma_{m',n'}$.

In~\cite{kpry:21}, Theorem 4.5, it is proved that if $r \geq s \geq 3, s \not = 4$, and $r \geq 6$, then $\gp{T_{r,s}}\in \{6, 7\}$.

\subsection{Complete bipartite graphs, cographs and more general graphs}

Let us start with a preliminary result about graphs such that almost all the vertices can be part of a \mvs.
\begin{lemma}\label{lem:V-1}
 Let $G=(V,E)$ be a graph. Then $\mu(G)\geq |V|-1$ if and only if there exists $v\in V$ adjacent to each vertex $u\in G-v$ such that $deg_{G-v}(u)<|V|-2$.
\end{lemma}

\begin{proof}
 $(\Rightarrow)$ If $\mu(G)=|V|$ then, by Lemma~\ref{lem:char}, $G$ is a clique graph and then the statement is obviously true.  If $\mu(G)=|V|-1$, then there exists a unique vertex $v$ of $G$ such that $v\not \in P$, where $P\in M(G)$. Let $u \in G-v$. If $deg_{G-v}(u)=|V|-2$, then $u$ is adjacent to any other point in $P$, and then is in mutual visibility with it. If $deg_{G-v}(u)<|V|-2$ then there exists at least a vertex $w \in P$ not adjacent to $u$. Since $u$ and $w$ are mutually visible, there must exist the path $(u,v,w)$, then $v$ is adjacent to $u$.
 
$(\Leftarrow)$ Let $P=V\setminus \{v\}$ be a set of points. Let us show that it is a \mvs and then that $ \mu(G)\geq |V|-1$. Let $u\in P$, if $deg_{G-v}(u)=|V|-2$, then, as noted above, $u$ is adjacent to any other vertex in $P$. Otherwise $deg_{G-v}(u)<|V|-2$ and $uv\in E$ by hypothesis. In this case, let $Q=P\setminus N_{G-v}[u]$ be the set of points in $P$ not adjacent to $u$. Then for each $w \in Q$, $deg_{G-v}(w)<|V|-2$. Hence $w$ is adjacent to $v$ and then $u$ and $w$ are in mutual visibility through the shortest path $(u,v,w)$. By the generality of $u$ and $w$, we have that $P$ is a \mvs of $G$.
\end{proof}

For complete bipartite graphs $K_{m,n}$, Table~\ref{tab:bip} reports the values of $\mu(K_{m,n})$ for small values of $m$ and $n$. 
Note that for $K_{2,n}\sim\overline{K_2}+\overline{K_n}$ Lemma~\ref{lem:V-1} applies if the vertex $v$ is taken in the partition $\overline{K_2}$.  A general result is the following.

\begin{table}[t]
\begin{center}
\begin{tabular}{c|c|c|c|c}
$m$&$n$&Graph $G$&$\mu(G)$&Reference\\\hline
1 &1 & $P_2$& 2&Lemma~\ref{lem:PnCn}\\\hline
1 &$n>1$& $K_{1,n}$&n&Corollary~\ref{cor:tree}\\\hline
2 & 2 & $C_4$ & 3&Lemma~\ref{lem:PnCn}\\\hline
2 & $n>2$& $K_{2,n}$&n+1&Lemma~\ref{lem:V-1}\\\hline
$m\geq 3$ & $n\geq 3$ &  $K_{m,n}$&n+m-2&Theorem~\ref{theo:Kmn}
\end{tabular}
\end{center}
 \caption{Values of $\mu(G)$ for $G\sim K_{m,n}$ for all the possible value of $m$ and $n$ such that $m\leq n$. }\label{tab:bip}
\end{table}

\begin{theorem}\label{theo:Kmn}
 Let $G$ be a complete bipartite graph $K_{m,n}$ such that $m\geq 3$ and $n\geq 3$. Then $\mu(G)=m+n-2$.
\end{theorem}
\begin{proof}
First we notice that $\mu(G)\leq m+n-2$ because $\mu(G)=m+n$ would imply that $G$ is a clique graph and $\mu(G)=m+n-1$ is not possible because Lemma~\ref{lem:V-1} does not apply. 
 To show that $\mu(G)=m+n-2$, it is sufficient to find exactly two vertices  that are not in the maximum \mvs consisting of all other vertices. Since $G\sim \overline{K_m} + \overline{K_n}$, we take a vertex $u$ from $
 \overline{K_m}$ and a vertex $v$ from $\overline{K_m}$. 
 
 Each point $w$ in $\overline{K_m}-v$ is in mutual visibility with other points in $\overline{K_m}-v$ because of vertex $u$. Furthermore, $w$ is adjacent to points in $\overline{K_n}-u$. Symmetrically, points in $\overline{K_n}-u$ are in mutual visibility because of vertex $v$ and are adjacent to all other points.
\end{proof}

We can generalize the results of Lemma~\ref{lem:V-1} and Theorem~\ref{theo:Kmn} to more general graphs resulting from a join operation.

\begin{corollary}\label{cor:join}
 Let $G_1=(V_1,E_1)$ and $G_2=(V_2,E_2)$ be two graphs and $J= G_1 + G_2=(V,E)$ their join. Then one of the following three cases holds:
 \begin{enumerate}
  \item $\mu(J)=|V| \iff G_1$ and $G_2$ are clique graphs
  \item $\mu(J)=|V|-1  \iff \mu(J)\not=|V|$  and $\mu(G_1)\geq |V_1|-1$ or $\mu(G_2)\geq |V_2|-1$
  \item $\mu(J)=|V|-2 \iff \mu(G_1)< |V_1|-1$ and $\mu(G_2)< |V_2|-1$.  
 \end{enumerate}
\end{corollary}
\begin{proof}
 \begin{enumerate}
  \item Obvious by Lemma~\ref{lem:char}.
  \item $(\Rightarrow)$ If $\mu(J)=|V|-1$ then there is a vertex $v$ that is not a point. Without loss of generality, let $v\in V_1$. Then each pair $x,y$ of non adjacent points in $G_1-v$ must be connected to $v$ to be in mutual visibility. Hence, by Lemma~\ref{lem:V-1}, $\mu(G_1)\geq|V|-1$.\\
  $(\Leftarrow)$ Without loss of generality, assume  $\mu(G_1)\geq |V_1|-1$.  If $\mu(G_1)= |V_1|-1$, let $v\in V_1$ be the only vertex of $G_1$ that is not a point, otherwise, if $G_1$ is a clique graph, let $v$ be any point of $V_1$. Given $\mu(G_1)\geq |V|-1$,  all the points in  $V_1\setminus \{v\}$ are in mutual visibility. Any pair of points in $V_2$ are in mutual visibility since either adjacent or connected to $v$. Since any point in $V_2$ is adjacent to any point in $V_1$, we conclude that all the points are in mutual visibility and then $\mu(J)=|V|-1$
  \item Let $v_1\in V_1$ and $v_2 \in V_2$, and let $P=V\setminus\{v_1,v_2\}$ be the set of points. Then $P$ is a \mvs, since any point in $V_1\setminus \{v_1\}$ ($V_2\setminus \{v_2\}$, resp.) is adjacent to any point in  $V_2$ ($V_1$, resp.) and it is in mutual visibility with any non adjacent point of $V_1$ with a shortest path of length two passing through $v_2$ ($v_1$, resp.).
 \end{enumerate}
\end{proof}

Cographs are well studied in literature and were independently rediscovered many times, since they represent the class of graphs that can be generated from $K_1$ by complementation and disjoint union (see Theorem 11.3.3 in~\cite{graph_classes_survey} for equivalent definitions).  As reported in Section~\ref{sec:notation}, a connected cograph can be obtained starting from $K_1$ by a sequence of splittings, that is by adding a sequence of twin vertices. In~\cite{GDS12} the notion of \emph{twin-free subgraph} was introduced.

\begin{definition}\emph{\hspace{-0.17cm}\cite{GDS12}}
Let $G=(V,E)$ be a graph. The \emph{twin-free subgraph} $\tf{G}$ of $G$ is the subgraph
$G[V']$ induced by the largest set of vertices $V'\subseteq V$ such that $G[V']$ has no twins.
\end{definition} 
 Since any induced subgraph of a connected cograph $G$ is a cograph, then $\tf{G}\sim K_1$ and it can be obtained by the polynomial time {\sc Pruning} algorithm presented in the same paper. This algorithm removes any vertex $v$ of $G$ that has a twin, and it applies the same procedure to $G-v$ until a graph without twin vertices is reached. Then it provides a sequence of vertex removals that corresponds to a sequence of splitting operations  to rebuild the whole $G$ starting from $K_1$ and in such a way that $G$ results the join of two of its subgraphs. Based on this observation we can provide the following result.

 \begin{theorem}
  Let $G=(V,E)$ be a connected cograph. Then $\mu(G)\geq |V|-2$ and a maximum \mvs can be computed in polynomial time.
 \end{theorem}
 \begin{proof}
 Let us show that the vertices of any connected cograph $G=(V,E)$ can be partitioned into two subsets $V_1$ and $V_2$ such that $G= G[V_1] +G[V_2]$.
 
 Let $v_1$ be the only vertex of $\tf{G}$ and let $V_1=\{v_1\}$. Since $G$ is connected, the first splitting operation to rebuild $G$ from $v_1$ produces a true twin $v_2$ of $v_1$ and the resulting graph is a $K_2$. Let $V_2=\{v_2\}$. Now add any vertex $v_1'$ ($v_2'$, resp.) produced by a splitting operation on a vertex of $V_1$ ($V_2$, resp.) to $V_1$ ($V_2$, resp.). Eventually, each vertex in $V_1$ is connected to all the vertices in $V_2$ and vice versa. Hence  $G= G[V_1] +G[V_2]$. By applying Corollary~\ref{cor:join}, $\mu(G)\geq |V|-2$. If all the splitting operations generate true twins, then $G$ is a clique graph and $\mu(G)=|V|$. By Algorithm~\mva, we can test if $V\setminus \{v\}$  is a \mvs of $G$ for some vertex $v\in V$, and then $\mu(G)=|V|-1$. Otherwise, $\mu(G)=|V|-2$ and $V\setminus \{v_1, v2\}$ is a \mvs of $G$.
\end{proof}

In~\cite{acckt:19}, Theorem 4.2, it is proved that if $G$ is a connected cograph, then $\gp{G}=\max\{\omega(G),\omega(\overline{G})\}$,



\section{Conclusions}\label{sec:concl}

This paper is a first study on the concepts of \mvs and \mvn. It would be interesting to study the same concepts for weighted graphs and directed graphs. The latter case is very different from the studied one since, given a set of points $P$, the relation of  visibility between two points is not symmetric. 

From a computational point of view, the \mvp problem could be analyzed with respect to  approximability and  parameterized complexity.

Given a graph, we have shown some relations between the \mvn and  both the general position number and the maximum degree of the graph. It would be interesting to study relations with other invariants, such as the treewidth or the clique-width of a graph. 

Finally, different kinds of visibility can be investigated. For example, the \emph{single point visibility}: find the vertex in the graph seen by the largest set of points.

\section*{Acknowledgments}
The author wish to thank Prof. Sandi Klav{\v z}ar for pointing out the concept of general position set in graphs and its relation with the concept of \mvs.\\
Figure~\ref{fig:block} was adapted from a drawing by David Eppstein, whose remarkable contributions to Wikipedia have greatly facilitated the writing of this article.



\bibliographystyle{abbrv}
\bibliography{references,Miei}

\end{document}